\documentclass[letterpaper,11pt]{amsart}
\usepackage[margin=1in]{geometry}
\usepackage{amsfonts}
\usepackage{pb-diagram}
\title{Alexander varieties and largeness of finitely presented groups}

\newtheorem{thm}{Theorem}
\newtheorem{lemma}[thm]{Lemma}

\newtheorem{cor}[thm]{Corollary}
\newtheorem{prop}[thm]{Proposition}

\numberwithin{thm}{section}

\newcommand{\al}{\alpha}

\newcommand{\gam}{\Gamma}

\newcommand{\bZ}{\mathbb{Z}}
\newcommand{\bR}{\mathbb{R}}
\newcommand{\bC}{\mathbb{C}}
\newcommand{\bN}{\mathbb{N}}

\DeclareMathOperator{\Hom}{Hom}

\DeclareMathOperator{\Mod}{Mod}

\DeclareMathOperator{\rk}{rk}

\DeclareMathOperator{\Tor}{Tor}

\newcommand{\yt}{\widetilde}
\newcommand{\ol}{\overline}
\newcommand{\yh}{\widehat}

\author{Thomas Koberda}
\address{Department of Mathematics, Harvard University, 1 Oxford St., Cambridge, MA 02138, USA}
\email{koberda@math.harvard.edu}

\begin{document}
\begin{abstract}
Let $X$ be a finite CW complex.  We show that the fundamental group of $X$ is large if and only if there is a finite cover $Y$ of $X$ and a sequence of finite abelian covers $\{Y_N\}$ of $Y$ which satisfy $b_1(Y_N)\geq N$.  We give some applications of this result to the study of hyperbolic $3$--manifolds, mapping classes of surfaces and combinatorial group theory.
\end{abstract}
\maketitle
\begin{center}
\today
\end{center}
\section{Introduction}
\subsection{Main results}
Let $X$ be a finite CW complex, or equivalently let $G=\pi_1(X)$ be a finitely presented group.  We say that $X$ is {\bf large} if there is a finite cover $Y\to X$ and a surjection $\pi_1(Y)\to F_2$, where here $F_2$ denotes a free group of rank two.  We say that $X$ has {\bf virtually infinite first Betti number} if for each $N$ there exists a finite cover $X_N\to X$ with $b_1(X_N)\geq N$.  The principal result of this paper is the following characterization of large finitely presented groups:

\begin{thm}\label{t:large}
Let $X$ be a finite CW complex.  Then the following are equivalent:
\begin{enumerate}
\item
$X$ is large.
\item
There is a finite cover $Y\to X$ such that for each $N$, there is a finite abelian cover $Y_N\to Y$ with $b_1(Y_N)\geq N$.
\item
There is a finite cover $Y\to X$ such that the Alexander variety of $Y$ contains infinitely many torsion points.
\end{enumerate}
\end{thm}

The Alexander variety $V_1(Y)\subset(\bC^*)^r$ of a finite CW complex $Y$ will be defined in Section \ref{s:back}.
If $X$ is large then clearly $X$ has virtually infinite first Betti number, though the converse is not true in general.  Theorem \ref{t:large} illustrates the degree to which the converse holds.
We now note some corollaries of Theorem \ref{t:large}.  Firstly, one can recover the following special case of a result of Cooper, Long and Reid in \cite{cooperlongreid}.  Cooper, Long and Reid give a proof of this result without the assumption that $M$ is fibered:

\begin{cor}\label{c:clr}
Let $M$ be a noncompact, finite volume, fibered, hyperbolic $3$--manifold.  Then $M$ is large.
\end{cor}

Let $S$ be an orientable surface with finite, negative Euler characteristic, and let $\Mod(S)$ denote its mapping class group.  Let $\psi\in\Mod(S)$ denote a nontrivial mapping class.  If $\yt{S}\to S$ is a finite, $\psi$--invariant cover, then $\psi$ admits a lift $\yt{\psi}$ to $\Mod(\yt{S})$.  In particular, we get an action $\yt{\psi}_*$ of $\yt{\psi}$ on $H_1(\yt{S},\bC)$.  We will write $\rho(\yt{\psi}_*)$ for the spectral radius of this action, which is to say the largest absolute value of an eigenvalue of $\yt{\psi}_*$.  Whereas $\yt{\psi}_*$ depends on the choice of lift of $\psi$, the value of $\rho(\yt{\psi}_*)$ does not.

Recall that each nontrivial mapping class $\psi$ of $S$ is either finite order, reducible or pseudo-Anosov, depending on whether or not $\psi$ has finite order and whether or not $\psi$ stabilizes the isotopy class of a nontrivial multicurve on $S$.  If $\psi\in\Mod(S)$, we write $T_{\psi}$ for the mapping torus.  It is well--known that the $3$--manifold $T_{\psi}$ is hyperbolic if and only if $\psi$ is pseudo-Anosov.

\begin{cor}\label{c:threeman}
Let $\psi\in\Mod(S)$ be a nontrivial mapping class.
\begin{enumerate}
\item
If $\psi$ is not pseudo-Anosov, then $T_{\psi}$ is large.
\item
If there is no lift $\yt{\psi}$ to a finite cover such that $\rho(\yt{\psi}_*)>1$, then $T_{\psi}$ is large.
\end{enumerate}
\end{cor}

We make a brief remark about part 2 of Corollary \ref{c:threeman}: the proof will show that if we suppose from the beginning that $b_1(T_{\psi})\geq 2$ then it suffices to assume that there is no lift $\yt{\psi}$ of $\psi$ to an abelian cover of $S$ such that $\rho(\yt{\psi}_*)>1$ in order to conclude largeness.  If $\rho(\psi_*)=1$ then there is a $k\geq 1$ such that $b_1(T_{\psi^k})\geq 2$.

We will also be able to use Theorem \ref{t:large} to recover the following result of Baumslag and Pride (see \cite{BP}):

\begin{cor}\label{c:bp}
Let $G$ be a group which admits a finite presentation \[\mathcal{P}=\langle S\mid R\rangle\] with $|S|-|R|\geq 2$.  Then $G$ is large.
\end{cor}

\subsection{Notes and references}
Virtually infinite first Betti number and largeness had been conjectured for general finite--volume hyperbolic orbifolds.  Lackenby's survey \cite{Lack2} contains both as conjectures 2 and 5 in the introduction and outlines the state of affairs at the time of the writing of that article, but these conjectures have been known since at least the work of Thurston in \cite{T3M}.  Both conjectures were closely connected with other central problems in $3$--manifold theory such as the virtually fibered and virtually Haken conjectures, both of which were stated in \cite{T3M}.  In recent work (see \cite{AGM}), Agol, Groves and Manning have announced a proof of all these conjectures, building on work of Wise (see \cite{wise1}, \cite{wise2}).
The proofs of the results in this article are independent of the work in \cite{AGM}, \cite{wise1} and \cite{wise2}.

The ideas behind the proof of Theorem \ref{t:large} fit into a more general discussion of polynomial periodicity.  A function $f:\bN\to\bZ$ is {\bf polynomial periodic} if there is a finite collection $\{p_0,\ldots,p_{n-1}\}$ of polynomials such that if $m\equiv i\pmod n$ then \[f(m)=p_{i}(m).\]  In \cite{Sarnak}, Sarnak studied the growth of Betti numbers of congruence covers of finite CW complexes, which is to say covers $X_N\to X$ induced by the natural map $\pi_1(X)\to H_1(X,\bZ/N\bZ)$.  The primary result of his work is that the first Betti numbers of the covers $\{X_N\}$ are polynomial periodic in $N$ (cf. Hironaka in \cite{Hiropp}).  Laurent's Theorem (\cite{Lau}) and Alexander theory allows us to conclude that if $X$ has unbounded growth of $b_1$ as we vary over finite abelian covers, then $X$ has at least linear growth of $b_1$ under finite abelian covers.  More careful analysis of various finite abelian covers of $X$ allows us to prove largeness, using a result of Lackenby in \cite{La}.

In \cite{button}, J. Button develops a criterion for largeness which also relies on Alexander theory, and he proves several result in combinatorial group theory.  Button's criterion provides a sufficient but not necessary condition for a finitely presented group to be large.

\section{Acknowledgements}
The author thanks B. Farb for many useful comments and suggestions.  The author is particularly grateful to C. McMullen for his help.  The results in this article owe a large intellectual debt to E. Hironaka.

\section{Background}\label{s:back}
\subsection{Eigenvalues of integer matrices}
In our study of homological actions of mapping classes, we will be appealing repeatedly to the following fact from number theory which is originally due to Kronecker.  We include a proof for the convenience of the reader:

\begin{prop}
Let $M\in GL_n(\bZ)$ and let $\lambda$ be an eigenvalue of $M$ of length one.  Suppose that $\lambda$ is not a root of unity.  Then the spectral radius of $M$ is greater than one.
\end{prop}
\begin{proof}
If $\lambda$ is any eigenvalue of $M$ then its minimal polynomial must divide the characteristic polynomial of $M$.  In particular, all the Galois conjugates of $\lambda$ are also eigenvalues of $M$.  It is well--known that if $\lambda$ is an algebraic integer of length one, all of whose Galois conjugates also have length one, then $\lambda$ is a root of unity.

Indeed, let $\{\al_1,\cdots,\al_k\}$ be the roots of the minimal polynomial of $\lambda$, each of which has norm one.  For each $n$, we define \[p_n(z)=\prod_{i=1}^k(z-\al_i^n).\]  The coefficients of $p_n(z)$ are symmetric functions of $\{\al_1,\cdots,\al_k\}$ and are therefore integral, and the coefficients are bounded since each $\al_i$ has length one.  It follows that the collection $\{p_n(z)\}$ is a finite collection of polynomials.  We have $p_m(z)=p_n(z)$ for some $n\neq m$.  Without loss of generality, $n$ divides $m$.  In particular, raising each $\al_i^n$ to an integral power permutes the list $\{\al_1^n,\ldots,\al_k^n\}$.  Replacing $m$ by a larger exponent if necessary shows that for each $i$, we have $\al_i^n=\al_i^m$.  In particular, $\al_i$ is a root of unity.
\end{proof}

\subsection{Alexander stratifications and Alexander polynomials}\label{s:alexpolys}
Let $G$ be a finitely presented group.  We will write \[G=\langle F_r\mid \mathcal{R}\rangle,\] where $F_r$ is a free group of rank $r$ and $\mathcal{R}$ is a finite collection of relations.  Write $\yh{G}$ for the group of characters of $G$, i.e. $\Hom(G,S^1)$.  The group $\yh{G}$ has the structure of an algebraic group whose coordinate ring is $\bC[G^{ab}]$.  Given a map $\al:G'\to G$, we get a map \[\yh{\al}:\yh{G}\to\yh{G'}\] by precomposition, and an induced map \[\yh{\al}^*:\bC[(G')^{ab}]\to\bC[G^{ab}].\]

If $g\in G$, write $\overline{g}$ for its image in $G^{ab}$, and
write $\Lambda_r(\bZ)$ for the ring of integral Laurent polynomial rings in $r$ variables over $\bZ$.  The Fox derivative furnishes maps \[D_i:F_r\to\Lambda_r(\bZ)\] which are defined by \[D_i(x_j)=\delta_{i,j}\] and by \[D_i(fg)=D_i(f)+\overline{f}D_i(g).\]  The $r$--tuple $(D_1,\ldots,D_r)$ is written $D$ and called the {\bf Fox derivative}.

Write $q$ for the natural surjection $F_r\to G$.  The {\bf Alexander matrix} of the presentation \[G=\langle F_r\mid \mathcal{R}\rangle\] is the matrix of partial derivatives \[M=M(F_r,\mathcal{R})=[(\yh{q})^*D_i(R_j)],\] where $R_j$ ranges over $\mathcal{R}$.

An alternative viewpoint on the Alexander matrix is as follows.  Let $X$ be a finite CW complex whose fundamental group $G$ with abelianization $\gam$, and let $Y$ be the corresponding cover of $X$.  We have a natural $\gam$--action on $Y$.  Choose a finite CW structure on $X$.  If $X$ has $s$ cells in dimension one and $r$ cells in dimension two, then there is a $\gam$--equivariant identification \[C_1(Y,\bC)\cong \bC[\gam]^s\] and \[C_2(Y,\bC)\cong\bC[\gam]^r.\]  The boundary map \[R:\bC[\gam]^r\to\bC[\gam]^s\] is represented by the Alexander matrix $M$, which can be computed from any presentation of $G$ via the Fox calculus.

If $G\to H$ is a surjection to another abelian group then one similarly obtains a boundary map \[R_H:\bC[H]^r\to\bC[H]^s.\]  The map $R$ is related to the map $R_H$ by the universal property of of the Alexander matrix:

\begin{thm}[See \cite{Hiropp}, Section 2]\label{t:universalprop}
Let $G$ be a finitely presented group and let \[\phi: G\to A\] be a surjective map to an abelian group $A$.  Then the matrix $\phi(M)$ represents the map $R_A$.
\end{thm}

We write $V_i(G)$ for the characters of $G$ such that the matrix $M(F_r,\mathcal{R})$ has rank less than $r-i$.  One obtains a stratification of $\yh{G}$ by algebraic subsets, called the {\bf Alexander stratification}: \[\yh{G}\supset V_1(G)\supset\cdots\supset V_r(G).\]  The meaning of the Alexander stratification is the following: a character $\chi\in\Hom(G,\bC^*)$ lies in $V_i(G)$ if and only if the dimension of the twisted homology space $H_1(G,\bC_{\chi})$ is at least $i$ (see \cite{Hir} and \cite{ctm4}).

If the abelianization of $G$ is torsion--free, one can define the {\bf Alexander polynomial} $A(G)$ as the greatest common divisor of the $(r-1)\times(r-1)$ minors of the Alexander matrix over the ring $\bZ[G^{ab}]$.  The Alexander polynomial defines the largest hypersurface contained in $V_1$.  If $G$ has torsion in its abelianization, one can project to the largest torsion--free quotient of $G$ and use the induced Alexander matrix guaranteed by Theorem \ref{t:universalprop}.  This will be the general definition of the Alexander polynomial.

By Theorem \ref{t:universalprop}, any surjection $\phi:G\to A$ to an abelian group which factors through the universal torsion--free abelian quotient of $G$ gives rise to a specialization of the Alexander polynomial under $\phi$, given by $\phi(A(G))\in\bZ[A]$.  Note that this definition makes sense even if $A$ has torsion.

The Alexander stratification is useful for computing the first Betti number of finite abelian covers of finite CW complexes:
\begin{thm}[\cite{Hir}, Proposition 2.5.6]\label{t:hiro}
Let $X$ be a finite CW complex such that $G=\pi_1(X)$, let $\al:G\to\gam$ be a map onto a finite abelian group, and let $X_{\al}$ be the finite cover of $X$ induced by $\al$.  Then \[b_1(X_{\al})=b_1(X)+\sum_{i=1}^r|V_i(G)\cap\yh{\al}(\yh{\gam}\setminus \yh{1})|.\]
\end{thm}

Thus the first Betti numbers of finite abelian covers of a finite CW complex can be computed from the torsion points in the Alexander strata.  Since the vanishing locus of the Alexander polynomial of $X$ is contained in $V_1$, any torsion points which are roots of the Alexander polynomial contribute to the first Betti number of finite covers of $X$.

For any algebraic subset $V$ of an affine torus, write $\Tor(V)$ for the torsion points contained in $V$.  The following result of Laurent (\cite{Lau}) is helpful for understanding the torsion points of an algebraic subset of an affine torus $(\bC^*)^r$:
\begin{thm}[\cite{Hir}, Theorem 4.1.1]\label{t:laurent}
If $V\subset (\bC^*)^r$ is any algebraic subset, then there exist rational planes $P_1,\ldots,P_k$ in $(\bC^*)^r$ such that each $P_i\subset V$ and such that \[\Tor(V)=\bigcup_{i=1}^k\Tor(P_i).\]
\end{thm}

If $\yt{S}$ is a finite cover of a surface $S$ with abelian Galois group $\gam$ then $H_1(\yt{S},\bC)$ splits into eigenspaces corresponding to the irreducible characters of the deck group.  For an irreducible character $\chi$, the $\chi$--eigenspace can be identified with the twisted homology group $H_1(S,\bC_{\chi})$.

Suppose that $b_1(T_{\psi})\geq 2$.  We write \[H=\Hom(H^1(S,\bZ)^{\psi},\bZ)\neq 0,\] where $H^1(S,\bZ)^{\psi}$ is the $\psi$--invariant cohomology of $S$.  We will be interested in finite covering spaces which arise from finite quotients of $H$.  If $\chi$ is a finite character of $H$, it turns out that the action of $\psi$ on $H_1(S,\bC_{\chi})$ is governed by the Alexander polynomial.

The group $G=H_1(T_{\psi},\bZ)/torsion$ decomposes as a direct sum $H\oplus\bZ$, where the $\bZ$--summand is dual to monodromy class $[\psi]$.  We will write $t$ for a generator of this $\bZ$.  The Alexander polynomial $A$ of $T_{\psi}$ with respect to the quotient $G$ of $\pi_1(T_{\psi})$ is an element of the Laurent polynomial ring $\bZ[G]$.

\begin{thm}[\cite{ctm4}, Corollary 3.2 and discussion immediately following]
Let $\chi$ be a character of $H$.  Then the characteristic polynomial of the action of $\psi$ on $H_1(S,\bC_{\chi})$ is given by $\chi(A)$.
\end{thm}

The meaning of $\chi(A)$ is that we substitute $\chi(h)$ for every element of $H$ occurring in $A$ and leave $t$ untouched.

\subsection{Lackenby's largeness criterion}
In order to show that certain groups are large, we will use the work of M. Lackenby in \cite{La}.  The relevant tools are {\bf homology rank gradient} and {\bf property ($\tau$)}.  To define these, we fix a prime $p$ and a sequence of nested finite index normal subgroups $\{G_i\}$ of a fixed group $G$.  We write $d(G_i)$ for the dimension of $H_1(G_i,\bZ/p\bZ)$ as a vector space over $\bZ/p\bZ$.  The homology rank gradient of the sequence $\{G_i\}$ is \[\gamma=\inf_i\frac{d(G_i)}{[G:G_i]}.\]  When $\gamma>0$, we say that the $\{G_i\}$ have {\bf positive modulo $p$ homology rank gradient}.

If $G$ is a finitely generated group with a finite generating set $S$ and $\{G_i\}$ is a collection of finite index subgroups, we write $X(G/G_i,S)$ for the coset graph of $G_i$ in $G$.  If $A$ is a set of vertices $V$ in a graph, we write $\partial A$ for the set of edges with exactly one vertex in $A$.  The {\bf Cheeger constant} $h(X)$ of a finite graph $X$ is given by \[h(X)=\min_{A\subset V,\, 0<|A|\leq|V|/2}\bigg\{\frac{|\partial A|}{|A|}\bigg\}.\]  $G$ has property ($\tau$) with respect to $\{G_i\}$ if the infimum of $h(X(G/G_i,S))$ is positive for some initial choice of $S$.  Lubotzky's book \cite{Lu} contains a detailed exposition on property ($\tau$).

Lackenby has proved many different results concerning conditions under which a group is large.  The one we will cite here can be found in \cite{La}:
\begin{thm}\label{t:lack}
Let $G$ be a finitely presented group.  Then the following are equivalent:
\begin{enumerate}
\item
$G$ is large.
\item
There exists a sequence of proper nested finite index subgroups \[G>G_1>G_2>\cdots\] and a prime $p$ such that:
\begin{enumerate}
\item
$G_{i+1}$ is normal in $G_i$ and has index a power of $p$ in $G_i$.
\item
$G$ does not have property ($\tau$) with respect to $\{G_i\}$.
\item
$\{G_i\}$ has positive modulo $p$ homology rank gradient.
\end{enumerate}
\end{enumerate}
\end{thm}

We begin by ruling out property ($\tau$) for our purposes.

\begin{lemma}\label{l:tau}
Let $G$ be finitely generated by a set $S$, and suppose that for each $i$ we have that $G/G_i$ is abelian.  Then $G$ does not have property ($\tau$) with respect to $\{G_i\}$.
\end{lemma}
\begin{proof}
Write \[G_{\infty}=\bigcap_i G_i.\]
We will show that the Cheeger constant will arbitrarily small as $i$ tends to infinity.  Note that in a Cayley graph for $G/G_{\infty}$ with respect to $S$, we may define the Cheeger constants of finite subgraphs by looking at the ratio of the size of the boundary of a finite set to the size of the set itself.  Since $G/[G,G]$ is amenable and surjects to $G/G_{\infty}$, the infimum of these Cheeger constants will be zero.  To see this, note that since the intersection of the subgroups $\{G_i\}$ is $G_{\infty}$, for any finite subset of the vertices in a Cayley graph for $G/G_{\infty}$ we may find an $i$ so that this set of vertices is mapped injectively to the set of vertices for the Cayley graph for $G/G_i$ with respect to $S$.  The degree of each vertex is non--increasing as we project the Cayley graph of $G/G_{\infty}$ to the Cayley graph of $G/G_i$.  It follows that the number of vertices in the boundary of a finite set of vertices in the Cayley graph of $G/G_{\infty}$ cannot increase under the projection map.  It follows that $\inf h(X(G,G_i,S))=0$.
\end{proof}

We thus obtain a corollary to Lackenby's result by combining Lemma \ref{l:tau} and the fact that for any finitely generated group $G$ we have \[\rk H_1(G,\bZ)\leq \rk H_1(G,\bZ/p\bZ):\]
\begin{cor}\label{c:lackcor}
Let $G$ be a finitely presented group, and let $\{G_i\}$ be a tower of nested, normal, $p$--power index subgroups of $G$ such that $G/G_i$ is abelian for each $i$.  Suppose that \[\inf_i\frac{\rk H_1(G_i,\bZ)}{[G:G_i]}>0.\]  Then $G$ is large.
\end{cor}

\section{Mapping tori with Magnus kernel monodromy}
In this section, we give the first application of Lackenby's criterion.  Recall that if \[S\to M\to S^1\] is a fibered $3$--manifold with monodromy $\psi$, the homology $H_1(M,\bZ)$ is given by $\bZ\oplus F$, where $F$ is the homology of $S$ which is co--invariant under the action of $\psi$.

Let $\widetilde{S}\to S$ be a finite characteristic cover and let $\psi\in\Mod(S)$, and suppose that $\psi$ commutes with the Galois group $G$ of the cover.  We obtain an induced covering $M'\to M$ as follows: the fundamental group of $M$ is given by a semidirect product \[1\to\pi_1(S)\to\pi_1(M)\to\bZ\to 1.\]  Since the Galois group $G$ commutes with the action of $\psi$, we obtain a homomorphism $\pi_1(M)\to G\times\bZ$, and by composing with the projection onto the first factor, a homomorphism $\pi_1(M)\to G$.

Note that the homology contribution of the base space $S^1$ is in the kernel of this homomorphism.  In particular, we obtain a covering map $M'\to M$ which lifts the identity map $S^1\to S^1$.  Furthermore, $\pi_1(M')$ fits into a short exact sequence \[1\to\pi_1(\widetilde{S})\to\pi_1(M')\to\bZ\to 1,\] where the conjugation action of $\bZ$ on $\pi_1(\widetilde{S})$ is the restriction of the action of $\psi$.  It follows that the rank of the homology of $M'$ is given by the rank of homology co--invariants of the action of $\psi$ on $H_1(\widetilde{S},\bZ)$ (equivalently the rank of the $\psi$--invariants of $H^1(\widetilde{S},\bZ)$).

We can now establish a fact about Magnus kernels.  Let $S^{ab}$ be the universal abelian cover of $S$.  Recall that the {\bf Magnus kernel} is the subgroup of the marked mapping class group which acts trivially on $H_1(S^{ab},\bZ)$.  It is not a priori clear that the Magnus kernel is nontrivial, but in fact it is infinitely generated (see \cite{CF}).

We first need to make the following observation:

\begin{lemma}
Let $\psi\in\Mod(S)$ be a mapping class contained in the Magnus kernel.  Then $\psi$ acts trivially on the integral homology of each finite abelian cover of $S$.
\end{lemma}
\begin{proof}
Since $\psi$ is in the Magnus kernel, then as an automorphism of $\pi_1(S)$, $\psi$ acts trivially on the universal metabelian quotient \[\pi_1(S)/[[\pi_1(S),\pi_1(S)],[\pi_1(S),\pi_1(S)]].\]  Suppose that $\widetilde{S}\to S$ is a finite abelian cover with Galois group $A$.  Then we have a short exact sequence \[1\to H_1(\widetilde{S},\bZ)\to M\to A\to 1.\]  Since $M$ is evidently metabelian, it is a quotient of \[\pi_1(S)/[[\pi_1(S),\pi_1(S)],[\pi_1(S),\pi_1(S)]].\]  It follows that $\psi$ restricts to the identity on $H_1(\widetilde{S},\bZ)$.
\end{proof}

\begin{prop}
Suppose that $\psi$ is contained in the Magnus kernel.  Then the fundamental group of the manifold $M_{\psi}$ is large.
\end{prop}
\begin{proof}
We suppose that $S$ is closed.  The proof in the non--closed case is analogous.  Let $p$ be any prime, $g>1$ the genus of $S$, and $H_i$ be the kernel of the map \[\pi_1(S)\to H_1(S,\bZ/p^i\bZ).\]  It is clear that $\{H_i\}$ forms a sequence of subgroups such that $\pi_1(S)/H_{i}$ is an abelian $p$--group for each $i$.  Furthermore, the rank of $H_1(S,\bZ/p^i\bZ)$ is $p^{2gi}$.  By an Euler character argument, we have that the genus of the $i^{th}$ surface $S_i$ corresponding to $H_i$ is $p^{2gi}(g-1)+1$.  It follows that the rank of $H_i^{ab}$ is $p^{2gi}(2g-2)+2$.

Now let $\psi$ be any mapping class in the Magnus kernel and let $T_{\psi}$ be suspension of $\psi$ with fundamental group $G$.  Since $\psi$ acts trivially on $H_1(S,\bZ)$, we have that $\psi$ commutes with $\pi_1(S)/H_i$ for each $i$ so that each $H_i$ gives us a finite cover $T_i$ of $T_{\psi}$ with fundamental group $G_i$.  Topologically, $T_i$ is just the suspension of $\psi$ as a mapping class of $S_i$.  Since $\psi$ is in the Magnus kernel, it follows that the rank of $G_i^{ab}$ is $p^{2gi}(2g-2)+2$.  Note furthermore that the index $[G:G_i]$ is equal to $[\pi_1(S):H_i]$, namely $p^{2gi}$.  Since $g>1$, the ratio between the rank of the homology of $G_i$ and $[G:G_i]$ is bounded away from zero.  Since $G/G_i$ is abelian for all $i$, we see that $G$ is large.
\end{proof}

\section{Alexander varieties and largeness}

We begin by relating the unbounded growth of $b_1$ over finite abelian covers to some aspects of the Alexander variety.

\begin{lemma}\label{l:alexinf}
Let $X$ be a finite CW complex and suppose that for each $N$, there is a finite abelian cover $X_N$ of $X$ with $b_1(X_N)\geq N$.  Then the Alexander variety $V_1(\pi_1(X))$ contains infinitely many torsion points.
\end{lemma}
\begin{proof}
Recall that if $\al:\pi_1(X)\to\gam$ is a surjection to a finite abelian group, then the first Betti number of the associated cover $X_{\al}$ satisfies \[b_1(X_{\al})=b_1(X)+\sum_{i=1}^r|V_i(\pi_1(X))\cap\yh{\al}(\yh{\gam}\setminus \yh{1})|\] (see Section \ref{s:alexpolys}).  Since $b_1(X_{\al})$ can be made arbitrarily large as $\al$ varies over all finite abelian quotients of $\pi_1(X)$ and since we have a sequence of inclusions \[V_r(\pi_1(X))\subset\cdots\subset V_1(\pi_1(X)),\] it follows that for each $N$ there is a finite abelian quotient \[\al_N:\pi_1(X)\to\gam_N\] for which \[|V_1(\pi_1(X))\cap\yh{\al_N}(\yh{\gam_N}\setminus \yh{1})|\geq N.\]  Since each element of $\yh{\al_N}(\yh{\gam_N}\setminus \yh{1})$ is a torsion point in $V_1$, the lemma follows.
\end{proof}

An immediate consequence of Laurent's Theorem and Lemma \ref{l:alexinf} is that there is a rational plane $P\subset V_1(\pi_1(X))$ whose intersection with $(S^1)^r$ has positive dimension.

\begin{lemma}\label{l:largechar}
Let $X$ be a finite CW complex and let $V_1(\pi_1(X))\subset (\bC^*)^r$ be its Alexander variety.  Suppose there is a rational plane $P\subset V_1(\pi_1(X))$ whose intersection with \[(S^1)^r\subset (\bC^*)^r\] has positive dimension.  Then $\pi_1(X)$ is large.
\end{lemma}
\begin{proof}
Clearly we may suppose that at least one component of $V_1(\pi_1(X))$ has positive dimension.  Since $V_1(\pi_1(X))$ sits inside of $\Hom(\pi_1(X),\bC^*)$, we may suppose that $b_1(X)\geq 1$.  Write $H_1(X,\bZ)\cong\bZ^m\oplus A$, where $m>1$ and where $A$ is a finite abelian group.  Choose a basis $\{w_1,\ldots,w_m\}$ for $H^1(X,\bZ)=\Hom(\pi_1(X),\bZ)$.  Composition of integral cohomology classes of $X$ with elements of $\Hom(\bZ,\bC^*)$ associates to each basis element $w_i$ a distinguished copy of $\bC^*$ inside of $(\bC^*)^m$.  Now note that $\yh{A}\cong A$, and that $\yh{A}$ consists of a finite set of points in $(S^1)^{r-m}$.  Thus, we may think of $\Hom(\pi_1(X),\bC^*)$ as sitting inside of \[(\bC^*)^m\times\yh{A}\subset (\bC^*)^r.\]  We will write $\{u_1,\ldots,u_m\}$ for the dual basis of $H_1(X,\bZ)/A$ determined by $\{w_1,\ldots,w_m\}$.  Write \[A=\bZ/n_1\bZ\oplus\cdots\oplus\bZ/n_{r-m}\bZ,\] where $n_i|n_{i+1}$, and let $\{x_1,\ldots,x_{r-m}\}$ be each identified with the corresponding generator $1\in\bZ/n_i\bZ$.

Choose a rational splitting $(\bC^*)^r$ which is compatible with the choices in the previous paragraph.  Since $P\cap (S^1)^r$ has positive dimension, there is a rational, one--dimensional subtorus of $(S^1)^r$ which is also contained in $V_1(\pi_1(X))$.  Such a subtorus is given (with respect to the chosen splitting) by a map $\bR\to(S^1)^r\subset(\bC^*)^r$ of the form \[\gamma:t\mapsto (\exp (2\pi i (a_1t+s_1)),\ldots,\exp (2\pi i (a_mt+s_m)),\exp(2\pi i s_{m+1}),\ldots,\exp(2\pi i s_r)),\] where $\{a_1,\ldots,a_m\}$ are integers which are not simultaneously zero, and where $\{s_1,\ldots,s_r\}$ are rational.  Consider the subgroup $\gam$ of $S^1$ generated by the roots of unity $\{\exp(2\pi i s_i)\}$, which has order $B$.  Choose an identification of the group $\gam$ with $\bZ/B\bZ$, and write $d_i\pmod B$ for the image of $\exp(2\pi is_i)$ under this identification.

Choose a prime $p$ which divides none of the nonzero $\{a_i\}$ nor $B$.  For each $n$, we consider the homomorphism $\al_n:\pi_1(X)\to\bZ/p^nB\bZ$ given by \[u_i\mapsto a_i+p^nd_i\pmod{p^nB}\] and by \[x_j\mapsto p^nd_j\pmod{p^nB}.\]  Let $\zeta$ be a primitive $p^nB^{th}$ root of unity, viewed as an element of $\yh{\bZ/p^nB\bZ}$.  The image of $\zeta$ under $\yh{\al_n}$ is just \[(\zeta^{a_1+p^nd_1},\ldots,\zeta^{a_m+p^nd_m},\zeta^{p^nd_{m+1}},\ldots,\zeta^{p^nd_r}).\]  Since $\zeta$ is a primitive $p^nB^{th}$ root of unity, $\zeta^{p^n}$ is a primitive $B^{th}$ root of unity.  Whenever $\zeta^{p^n}$ is equal to the generator $1\in\bZ/B\bZ$ via the inverse of the identification of $\gam$ with $\bZ/B\bZ$, we have that $\zeta^{p^nd_i}$ is equal to the root of unity $\exp(2\pi i s_i)$.  There are $B$ distinct $B^{th}$ roots of unity, so for at least $1/B$ of the primitive $p^nB^{th}$ roots of unity, we will have that $\zeta^{p^nd_i}$ is equal to $\exp(2\pi i s_i)$.  Whenever this is the case, we have that the point \[(\zeta^{a_1+p^nd_1},\ldots,\zeta^{a_m+p^nd_m},\zeta^{p^nd_{m+1}},\ldots,\zeta^{p^nd_r})\] lies on the subtorus $\gamma$.  Writing $X_n$ for the cover corresponding to $\al_n$, the formula \[b_1(X_{\al})=b_1(X)+\sum_{i=1}^r|V_i(\pi_1(X))\cap\yh{\al}(\yh{\gam}\setminus \yh{1})|\] implies that \[b_1(X_n)\geq \frac{1}{B}p^nB-1=p^n-1.\]

Let $k$ be the multiplicative order of $p$ modulo $B$.  We claim that $X_{n+k}$ covers $X_n$.  Since $B$ and $p$ are relatively prime, we have a splitting for each $n$: \[\bZ/p^nB\cong\bZ/p^n\bZ\oplus\bZ/B\bZ,\] and we have projections $q_1$ and $q_2$ onto the two factors.  The map $\al_{n+k}$ is given by taking \[u_i\mapsto a_i+p^{n+k}d_i\pmod{p^{n+k}B}\] and \[x_j\mapsto p^{n+k}d_j\pmod{p^{n+k}B}.\]  Considering the images of $\{u_i\}$ and $\{x_j\}$ in \[\bZ/p^nB\cong\bZ/p^n\bZ\oplus\bZ/B\bZ,\] we see that the image of $p^k(p^nd_i)$ is equal to the image of $p^nd_i$.  Indeed, $d_i$ can be written as $(q_1(d_i),q_2(d_i))$.  Since the first coordinate becomes zero when multiplied by $p^n$, we see that \[p^{n+k}(q_1(d_i),q_2(d_i))=(0,p^{n+k}q_2(d_i))=(0,p^nq_2(d_i))=p^n(q_1(d_i),q_2(d_i)).\]  It follows that $\al_n$ is equal to the reduction of $\al_{n+k}$ modulo $p^nB$.

Thus, we have a tower of covers \[\cdots\to X_{2k}\to X_k\to X,\] where $X_k\to X$ has degree $p^kB$, and where $X_{(n+1)k}\to X_{nk}$ has degree $p^k$ when $n>0$.  In particular $\{X_{nk}\}_{n>0}$ is a tower of abelian $p$--power covers with \[b_1(X_{nk})\geq \frac{1}{B}p^{nk}B-1=p^{nk}-1.\]  It follows that this tower has positive homology rank gradient, so that by Corollary \ref{c:lackcor}, $X$ is large.
\end{proof}

Finally, we may give a proof of the characterization of large groups:
\begin{proof}[Proof of Theorem \ref{t:large}]
Let $\gam$ be a large group.  Then there is a finite index subgroup $\gam'<\gam$ which surjects to the nonabelian free group $F_2$.  Choosing a sequence of surjective homomorphisms \[\phi_N:F_2\to\bZ\to\bZ/N\bZ,\] we have that $b_1(ker(\phi_N))\geq N+1$.  Precomposing $\phi_N$ with a surjection $\gam'\to F_2$, it follows that for each $N$, there is a finite index subgroup $\gam'_N<\gam$ such that $\gam'/\gam'_N$ is abelian and such that $b_1(\gam'_N)\geq N+1$.  Thus (1) implies (2).  (2) implies (3) is a consequence of Lemma \ref{l:alexinf} and the remark immediately following the proof.  (3) implies (1) is a consequence of Lemma \ref{l:largechar}.
\end{proof}

\section{Applications of Theorem \ref{t:large}}
We can now give proofs of the corollaries to Theorem \ref{t:large} which were mentioned in the introduction.

\subsection{Non--compact hyperbolic $3$--manifolds}
\begin{cor}[Cf. Corollary \ref{c:clr}]
Let $T_{\psi}$ be a fibered hyperbolic $3$--manifold with at least one cusp.  Then $T_{\psi}$ is large.
\end{cor}
\begin{proof}
Let $S$ be a fiber of $T_{\psi}$.  Then $S$ has at least one puncture.  Choose a finite cover $S'$ of $S$ to which $\psi$ lifts and which has at least three punctures.  Replacing the lift $\yt{\psi}$ by some positive power if necessary, we may assume that the punctures of $S'$ are preserved point--wise.  Let $X$ denote suspension of the lift of $\yt{\psi}$ acting on $S'$.

Let $[\gamma]$ denote the homology class of a small loop about one of the punctures of $S'$.  Clearly, $[\gamma]$ is invariant under the action of $\yt{\psi}$.  There is a sequence of finite, abelian, $\yt{\psi}$--invariant covers of $S'$ which ``unwind" $\gamma$.  For each $N$, we can find such a cover $S_N$ such that the number of punctures is at least $N$.  Replacing $\yt{\psi}$ by a further power will stabilize these punctures pointwise, resulting in a finite abelian cover of $X$ whose first Betti number is at least $N$.  It follows that $X$ is large.
\end{proof}

\subsection{Mapping tori}
Let $\psi\in\Mod(S)$ and let $T_{\psi}$ be the associated mapping torus.  We are ready to give the proof of Corollary \ref{c:threeman}.  Let us recall the statement:

\begin{cor}
Let $\psi\in\Mod(S)$ be any mapping class.
\begin{enumerate}
\item
If $\psi$ is not pseudo-Anosov, then $T_{\psi}$ is large.
\item
If there is no lift $\yt{\psi}$ to a finite cover such that $\rho(\yt{\psi}_*)>1$, then $T_{\psi}$ is large.
\end{enumerate}
\end{cor}
\begin{proof}
Suppose that $\psi$ is trivial or has finite order.  Then there is an $N$ such that $\psi^N$ is trivial.  Then $\pi_1(T_{\psi^N})\cong\bZ\times \pi_1(S)$.  Since $\chi(S)<0$, we have that $\pi_1(S)$ is large and therefore $T_{\psi}$ is large.

Suppose $\psi$ has infinite order and is reducible.  Write $\mathcal{C}$ for the canonical reduction system of $\psi$.  It is easy to see that there is a cover $\yt{S}$ of $S$ to which $\psi$ lifts and such that the homology classes of the components of the total lift of $\mathcal{C}$ span a subspace of $H_1(\yt{S},\bZ)$ of dimension at least two.  Replacing a lift $\yt{\psi}$ of $\psi$ to $\yt{S}$ by a power, we may assume that the components of the total lift of $\mathcal{C}$ are fixed.  Choosing the cover $\yt{S}\to S$ to be characteristic and choosing the power of $\yt{\psi}$ to be large enough, we may assume that the suspension of $\yt{\psi}$ acting on $\yt{S}$ is a cover $X$ of $T_{\psi}$.  Let $\gamma$ be a nonseparating component of the total lift of $\mathcal{C}$ to $\yt{S}$, and let $\al$ be another nonseparating component of the lift of $\mathcal{C}$ whose homology class is independent of that of $\gamma$.  Cut open $\yt{S}$ along $\gamma$.  For each $N$, let $S_N$ be the cover of $\yt{S}$ given by stringing together $N$ copies of $\yt{S}\setminus\gamma$ cyclically.  For each $N$, we have that the Galois group of the cover $S_N\to\yt{S}$ commutes with $\yt{\psi}$.  Replacing $\yt{\psi}$ by a power, each of the $N$ components of the total lift of $\al$ to $S_N$ will be fixed.  Thus, we obtain a family of abelian covers of $X$ whose first Betti numbers are arbitrarily large.  By Theorem \ref{t:large}, we have that $T_{\psi}$ is large.

Finally, suppose that $\psi$ is pseudo-Anosov, and assume that $\rho(\yt{\psi}_*)=1$ for each lift of $\psi$ to a finite cover of $S$.  Since $\rho(\psi_*)=1$, there is a $k>0$ such that $b_1(T_{\psi^k})\geq 2$.  Equivalently, $\Hom(\pi_1(S),\bZ)^{\psi^k}$ is nontrivial.  Fix a $\psi^k$--invariant cohomology class $\phi$, and let $\phi_N$ be the map $\pi_1(S)\to\bZ/N\bZ$ given by reducing the image of $\phi$ modulo $N$.  For each $N$, we will write $S_N$ for the corresponding cover of $S$.  For each irreducible character $\chi$ of $\bZ/N\bZ$, we may consider the $\chi$--eigenspace of $H_1(S_N,\bC)$, which we can identify with the twisted homology space $H_1(S,\bC_{\chi})$.  Since $\rho(\yt{\psi}_*)=1$ for all lifts of $\psi$ to finite covers of $S$, we have that some further nonzero power $\psi^{k'}$ of $\psi^k$ fixes a vector in $H_1(S,\bC_{\chi})$.  Since there are $N$ irreducible representations of $\bZ/N\bZ$, it follows that there is a nonzero exponent $k_N$ such that the suspension $\yt{\psi}^{k_N}$ acting as a mapping class of $S_N$ has at least $N$ linearly independent fixed vectors.  It follows that $T_{\yt{\psi}^{k_N}}$ has first Betti number at least $N+1$.  Notice that $T_{\yt{\psi}^{k_N}}$ is a finite abelian cover of the manifold $T_{\psi^k}$.  It follows that for each $N$, the manifold $T_{\psi^k}$ has a finite abelian cover with first Betti number at least $N$.  By Theorem \ref{t:large}, it follows that $T_{\psi}$ is large.
\end{proof}

\subsection{Combinatorial group theory}
In this subsection, we show that Theorem \ref{t:large} can be used to recover the Baumslag--Pride Theorem, namely that a group which admits a presentation with at least two more generators than relations is large.  Baumslag and Pride proved their result in \cite{BP}.  Ab\'ert has shown in \cite{Abert} that a Baumslag--Pride group itself may not surject onto a free group, and the author gave a bound on the index of a subgroup which surjects onto a free group in \cite{Koblarge}.

We begin by recalling the Reidemeister--Schreier rewriting process (we copy \cite{LySch} nearly verbatim).  Let $G=F/N$ be a group, where $F$ is free with free basis $X$, and where $N$ is the normal closure of a subset $R\subset F$.  Write $\phi$ for the canonical projection from $F$ to $G$, and let $H<G$ be a subgroup.  Write $\yt{H}$ for the preimage of $H$ under $\phi$, and let $T$ be a transversal for $\yt{H}$ in $F$.  For $w\in F$, we write $\ol{w}$ for the unique element of $T$ which satisfies \[Hw=H\ol{w}.\]  For $t\in T$ and $x\in X$, write \[\gamma(t,x)=tx(\ol{tx})^{-1},\,\,\gamma(t,x^{-1})=tx^{-1}(\ol{tx^{-1}})^{-1}=\gamma(tx^{-1},x)^{-1}.\]  Then $H$ admits a presentation \[H=\langle X^*\mid R^*\rangle,\] where $X^*$ consists of elements $\{\gamma(t,x)^*\}$ in one--to--one correspondence with all nontrivial elements of the form $\gamma(t,x)$ for $x\in X$ and $t\in T$, and $R^*$ is defined as follows: writing \[w=y_1\cdots y_n,\] where each letter appearing on the right hand side is in $X$, we write \[\tau(w)=\gamma(1,y_1)^*\gamma(\ol{y_1},y_2)^*\cdots\gamma(\ol{y_1\cdots y_{n-1}},y_n)^*.\]  Then $R^*$ consists of all the elements $\{\tau(trt^{-1})\}$, where $t\in T$ and $r\in R$.

We will prove Corollary \ref{c:bp} in the special case where $G$ has three generators and one relation.  The proof in general will not be difficult to formulate.  The reader may note some similarities between the beginnings of the proof here and the original proof in \cite{BP}.

\begin{proof}[Proof of Corollary \ref{c:bp}]
We write \[G=\langle t,a,b\mid w\rangle.\]  By applying an automorphism of $F_3$, we may assume that the exponent sum of $t$ in $w$ is zero.  For each $N$, we will write $G_N$ for the kernel of the map \[\phi_N:G\to\bZ/N\bZ\] which sends $t$ to the generator $1$ and both $a$ and $b$ to zero.  Since the exponent sum of $t$ in $w$ is zero, this homomorphism is defined.  We will write $T=\{1,t,\ldots,t^{N-1}\}$ for a Schreier transversal for $G_N$.  By the Reidemeister--Schreier rewriting process, we see that $G_N$ admits a presentation $\langle X_N\mid R_N\rangle$, where by construction, $R_N$ contains at most $N$ elements.  The set $X_N$ is in bijective correspondence with nonidentity elements of the form $\gamma(s,x)$, where $s\in T$ and $x\in\{t,a,b\}$.  For each $s\in T$, it is trivial to check that $\ol{sa}=s$ and that $\ol{sb}=s$, since both $a$ and $b$ are in the kernel of $\phi_N$.  Thus, for each $s\in T$, we have that \[\gamma(s,a)=sas^{-1},\] and similarly after replacing $a$ by $b$.  It follows that $X_N$ consists of at least $2N$ elements.  Passing to the abelianization of $G_N$, we see that the rank of $H_1(G_N,\bZ)$ is at least $N$.  The corollary follows by Theorem \ref{t:large}.
\end{proof}

\end{document}